\documentclass[12pt]{article}
\usepackage{latexsym,amsmath,amssymb}

\newif\ifdviwin

\usepackage[english]{babel}
\usepackage{indentfirst}
\usepackage[mathscr]{eucal}
\usepackage{amssymb,amsmath,amsfonts}
\usepackage{fancybox,fancyhdr}
\usepackage{graphicx}
\usepackage[utf8]{inputenc}
\usepackage{float}
\usepackage{color}
\usepackage[pdftex]{hyperref}
\hypersetup{colorlinks=true,linkcolor=blue,citecolor=blue} 

\newif\ifdviwin

\dviwintrue

\def\cS{\mathcal{S}}

\def\m2r{\mathbb{M}^2\times\mathbb{R}}
\def\h2r{\mathbb{H}^2\times\mathbb{R}}

\let\8=\infty \let\0=\emptyset

\let\tilde=\widetilde

\let\parc=\partial

\def\esiz{\langle}
\def\esde{\rangle}

\def\cte.{\mathop{\rm cte.}\nolimits}

\def\E{\mathbb{E}}

\def\R{\mathbb{R}}

\def\M{\mathbb{M}}

\def\D{\mathbb{D}}

\def\H{\mathbb{H}}
\def\S{\mathbb{S}}

\def\Ekt{\E(\kappa,\tau)}
\def\nil{\mathrm{Nil}_3}
\def\psl{\tilde{PSL_2}(\R)}

 \newtheorem{defi}{Definition}
 \newtheorem{teo}[defi]{Theorem}
 
 \newtheorem{cor}[defi]{Corollary}
 \newtheorem{lem}[defi]{Lemma}

 \newenvironment{proof}{\rm \trivlist \item[\hskip \labelsep{\it
      Proof}:]}{\nopagebreak \hfill $\Box$ \endtrivlist}

\numberwithin{equation}{section}

\begin{document}

\mbox{}\vspace{0.4cm}

\begin{center}
\rule{15cm}{1.5pt}\vspace{0.5cm}

\renewcommand{\thefootnote}{\,}
{\Large \bf Height estimates for constant mean curvature graphs in $\mathrm{Nil}_3$ and $\tilde{PSL_2}(\mathbb{R})$ \footnote{\hspace{-.75cm} Mathematics Subject
Classification: 53A10,53C30}}\\ \vspace{0.5cm} {\large Antonio Bueno}\\ \vspace{0.3cm} \rule{15cm}{1.5pt}
\end{center}
  \vspace{1cm}
Departamento de Geometría y Topología, Universidad de Granada,
E-18071 Granada, Spain. \\ e-mail: jabueno@ugr.es \vspace{0.3cm}

 \begin{abstract}
In this paper we obtain height estimates for compact, constant mean curvature vertical graphs in the homogeneous spaces $\mathrm{Nil}_3$ and $\widetilde{PSL}_2(\mathbb{R})$. As a straightforward consequence, we announce a structure-type result for complete graphs defined on relatively compact domains.
 \end{abstract}

\section{Introduction}

In the last decades, height estimates have become a powerful tool when studying the global behavior of a certain class of immeresd surfaces in some ambient space, see for instance \cite{AEG,He,HLR, KKMS,KKS, Me, Ro1,RoSa}. Heinz \cite{He} proved that if $M$ is a \emph{compact graph} in the euclidean space $\R^3$ with positive constant mean curvature $H$, ($H$-surface in the following), and boundary $\partial M$ lying on a plane $\Pi$, then the \emph{maximum height} that $M$ can reach from $\Pi$ is $1/H$. This estimate is optimal, since it is attained by the $H$-hemisphere intersecting orthogonally $\Pi$. Applying \emph{Alexandrov reflection technique} yields that a compact embedded $H$-surface in $\R^3$ with boundary on $\Pi$ has height from $\Pi$ at most $2/H$.

These height estimates for $H$-surfaces in $\R^3$ were the cornerstone for Meeks \cite{Me} in his global study of $H$-surfaces in $\R^3$; for example, he showed that there do not exist properly embedded $H$-surfaces with one end in $\R^3$, and if a properly embedded $H$-surface has two ends, then the surface stays at bounded distance from a straight line. Later, Korevaar, Kusner and Solomon \cite{KKS} proved that a properly embedded $H$-surface lying inside a solid cylinder must be rotationally symmetric and hence a cylinder or an onduloid.  

In \cite{KKMS} Korevaar, Kusner, Meeks and Solomon obtained optimal bounds for the height of $H$-graphs and of compact, embedded $H$-surfaces when $H>1$ in the \emph{hyperbolic space} $\H^3$ with boundary lying on a totally geodesic plane. In the formulation of the problem in both \emph{space forms} $\R^3$ and $\H^3$, the plane where $\partial M$ lies can be chosen without specifying its orthogonal direction, as $\R^3$ and $\H^3$ are \emph{isotropic}; in general, a riemannian manifold is isotropic if its isometry group acts transitively on the tangent bundle. 

In this context, the product spaces $\Sigma^2\times\R$ defined as the riemannian product of a complete riemannian surface $\Sigma^2$ and the real line $\R$ are closely related to the space forms in the sense that they are highly symmetric. In \cite{HLR} Hoffman, de Lira and Rosenberg obtained height estimates for compact embedded $H$-surfaces with boundary contained in a slice $\Sigma^2\times\{t_0\}$. This result was improved by Aledo, Espinar and Gálvez \cite{AEG}, exhibiting \emph{sharp} bounds for the height of compact, embedded $H$-surfaces in $\Sigma^2\times\R$ with boundary in a slice, and characterizing when equality held. As happened in $\R^3$ and $\H^3$, the $H$-graph in $\Sigma^2\times\R$ with boundary in a slice $\Sigma^2\times\{t_0\}$ attaining the maximum height over $\Sigma^2\times\{t_0\}$, corresponds to the rotational $H$-hemisphere intersecting orthogonally $\Sigma^2\times\{t_0\}$. 

For the particular case when the base $\Sigma^2$ is a complete, simply connected surface with constant curvature $\kappa$, the spaces arising are the product spaces $\M^2(\kappa)\times\R$. Such product spaces belong to a two parameter family of homogeneous, simply connected 3-dimensional manifolds, the $\Ekt$ spaces. In Section \ref{prel}, we will introduce these spaces and give a geometric sense to the constants $\kappa$ and $\tau$. For instance, the product spaces correspond to the case $\tau=0$ in the $\Ekt$ family. In the last decade, the theory of immersed surfaces in the $\Ekt$ spaces, and more specifically constant mean curvature and minimal surfaces, have become a fruitful theory focusing the attention of many geometers. See \cite{AbRo1, AbRo2,Da, DHM, FeMi} and references therein for an outline of the development of this theory.

Our objective in this paper is to obtain height estimates for vertical $H$-graphs in the Heisenberg space $\nil$ and in the space $\psl$, the universal cover of the positively oriented isometries of the hyperbolic plane $\mathbb{H}^2$, which correspond to the particular choices in the $\Ekt$ family of $\kappa=0,\tau>0$ and $\kappa<0,\tau>0$, respectively. To obtain the height estimates we will use the fact that $H$-graphs (or more generally, $H$-surfaces transverse to a Killing vector field) are stable, hence have bounded curvature at any fixed positive distance from their boundary. This behavior of the stability of $H$-surfaces has been exploited widely in the literature; see the proof of the Main Theorem in \cite{RST} for a global understanding of this technique in arbitrary complete 3-manifolds with bounded sectional curvature.

\section{Homogeneous 3-dimensional spaces with 4-dimen-\\sional isometry group.}\label{prel}

Let $\M^2(\kappa)$ be the complete, simply connected surface of constant curvature $\kappa\in\R$. The family of homogeneous, simply connected 3-dimensional manifolds $\E$ with a 4-dimensional isometry group, can be defined as a family of riemannian submersions $\pi:\E\longrightarrow\M^2(\kappa)$. The \emph{fibre} that passes through a point $p\in\M^2(\kappa)$ is defined as $\pi^{-1}(p)$, and translations along these fibres are ambient isometries generated by the flow of a unitary Killing vector field, $\xi$. The Killing vector field is related to the Levi-Civita connection $\overline{\nabla}$ of $\E$ and the cross product by the formula
$$
\overline{\nabla}_X\xi=\tau X\times\xi,
$$
where $\tau$ is a constant named the \emph{bundle curvature}. Both $\kappa$ and $\tau$ satisfy $\kappa-4\tau^2\neq 0$, and after a change of orientation of $\E$ we can suppose that $\tau>0$. These spaces are denoted by $\Ekt$, where $\kappa,\tau$ are the constants defined above. Depending on the value of $\kappa$ and $\tau$, we obtain the following geometries:

\begin{itemize}
\item If $\tau=0$, then we have the product spaces $\M^2(\kappa)\times\R$, i.e. the space $\S^2(\kappa)\times\R$ if $\kappa>0$, and the space $\H^2(\kappa)\times\R$ if $\kappa<0$.

\item If $\tau>0$ and $\kappa=0$, the $\Ekt$ space arising is the Heisenberg group $\nil$, the Lie group of matrices 
$$
\left\{\left(\begin{matrix}
1&a&b\\
0&1&c\\
0&0&1
\end{matrix}\right);\ \ a,b,c\in\R\right\},
$$
endowed with a one-parameter family of left-invariant metrics.

\item When $\tau>0$ and $\kappa<0$, we obtain the space $\psl$, the universal cover of the positively oriented isometries of the hyperbolic plane $\H^2$, endowed with a two-parameter family of left-invariant metrics. Up to a homothetic change of coordinates, the family of left-invariant metrics turns out to depend on one parameter.

\item When $\tau>0$ and $\kappa>0$, the $\Ekt$ spaces are the Berger spheres. These spaces can be realized as the 3-dimensional sphere $\S^3$ endowed with a one-parameter family (again, after a homothetic change) of metrics, which are obtained in such a way that the Hopf fibration is still a riemannian fibration, but the length of the fibres is modified.
\end{itemize}
We can give a unified model for the $\Ekt$ spaces; when $\kappa\leq 0$ the model is global and when $\kappa>0$ we get the universal cover of $\Ekt$ minus one fibre. We endow $\R^3$ (if $\kappa\geq 0$) and $\big(\D(2/\sqrt{-\kappa})\times\R\big)$ (if $\kappa<0$) with the metric
$$
ds^2=\lambda^2\big(dx^2+dy^2\big)+\big(dz+\tau\lambda(ydx-xdy)\big)^2,
$$
where $\lambda$ is defined as
$$
\displaystyle{\lambda=\frac{4}{4+\kappa(x^2+y^2)}}.
$$
The riemannian submersion is given by the projection onto the first two coordinates. The vector field $\partial_z$ is the unitary Killing vector field whose flow generates the \emph{vertical translations}. The integral curves of this flow are the fibres of the submersion, and they are complete geodesics. The fields given by
$$
E_1=\frac{1}{\lambda}\partial_x-\tau y\partial_z,\ \ E_2=\frac{1}{\lambda}\partial_y+\tau x\partial_z,\ \ E_3=\partial_z,
$$
are an orthonormal basis at each point. In this framework, the angle function of an immersed, orientable surface $M$ is defined as $\nu=\langle\eta,\partial_z\rangle$, where $\eta$ is a unit normal vector field defined on $M$.

Henceforth, we will denote simply by $\big(\mathbb{E},\langle\cdot,\cdot\rangle\big)$ to any of the $\Ekt$ spaces with the model given above. A \emph{section}\footnote{Abresch and Rosenberg also call these surfaces \emph{umbrellas}, see \cite{AbRo2}.} of $\mathbb{E}$ is a subset of the form $\{z=z_0;\ z_0\in\R\}$, where $z_0$ is called the height of the section. Every such a section is a minimal surface, and when $\tau=0$ they are totally geodesic copies of $\M^2(\kappa)$ that differ one from the other by a vertical translation. A \emph{vertical graph} in $\E$ is a surface with the property that it intersects each fibre of the submersion at most once. As a matter of fact, each vertical graph in $\E$ can be parametrized as
$$
\big\{(x,y,f(x,y));\ (x,y)\in\Omega\big\},
$$
for a certain smooth function $f$ defined in a domain $\Omega$ contained in some section $\{z=z_0\},\ z_0\in\R$. Note that after a vertical translation, the domain of a vertical graph can be contained in a section with any height. A graph is compact if $\Omega$ is compact and $f$ extends to $\partial\Omega$ continuously. The \emph{boundary} of a compact graph is defined as $f(\partial\Omega)$. A compact graph has boundary in a section if its boundary has constant height. This is equivalent to the fact that $f$ restricted to $\partial\Omega$ is a constant function.

\subsection{Stability of $H$-surfaces in the $\Ekt$ spaces}\label{estabilidad}

It is a well known fact that an $H$-surface $M$ immersed in an arbitrary riemannian 3-manifold, is a critical point for the area functional associated to compactly supported variations of the surface that preserve the enclosed volume. Equivalently, $M$ is an $H$-surface if and only if it is a critical point for the functional Area-$2H$Vol  \cite{BCE}. The second variation of this functional is given by the quadratic form
\begin{equation}\label{eqestabilidad}
\mathcal{Q}(f,f)=-\int_M\big(\Delta_M f+f(|\sigma|^2+\mathrm{Ric}(\eta)\big)fdA,\hspace{.5cm} \forall f\in C_0^\infty(M),
\end{equation}
where $\Delta_M$ is the Laplace-Beltrami operator of the surface $M$, $|\sigma|^2$ is the squared length of the second fundamental form of $M$, $\eta$ is the unit normal of $M$ and $\mathrm{Ric}(\eta)$ is the \emph{Ricci curvature} along the direction $\eta$. Equation \eqref{eqestabilidad} can be rewritten by defining the elliptic operator
\begin{equation}\label{jacobi}
\mathcal{L}=\Delta_M+|\sigma|^2+\mathrm{Ric}(\eta)
\end{equation}
and thus \eqref{eqestabilidad} is equivalent to
\begin{equation}\label{estabilidadjacobi}
\mathcal{Q}(f,f)=-\int_M f\mathcal{L}f dA,\hspace{.5cm} \forall f\in C_0^\infty(M).
\end{equation}
The operator $\mathcal{L}$ is the \emph{Jacobi operator}, or \emph{stability operator} of $M$. An orientable immersion $M$ in an $\Ekt$ space is said to be \emph{stable} if and only if
$$
-\int_M f\mathcal{L}f dA\geq 0,\hspace{.5cm} \forall f\in C_0^\infty(M).
$$

The non-vanishing functions $f\in C^\infty(M)$ lying in the kernel of $\mathcal{L}$ are called \emph{Jacobi functions}. If $M$ is an orientable immersed surface in an $\Ekt$ space and $\nu$ denotes the angle function of $M$, then $\nu$ is a Jacobi function for the stability operator $\mathcal{L}$ \cite{Da}, i.e. the elliptic equation $\mathcal{L}\nu=0$ holds. This equation reads as
\begin{equation}\label{pdeangle}
\Delta_M\nu+\nu\big((1-\nu^2)(\kappa-4\tau^2)+|\sigma|^2+2\tau^2\big)=0.
\end{equation}
A classical theorem due to Fischer-Colbrie \cite{Fi} asserts that the existence of a positive Jacobi function defined on a surface $M$ is equivalent to the stability of the surface. 

Consider now a vertical graph $M$ in $\Ekt$. As by definition $M$ intersects each fibre of the space $\Ekt$ at most once, then $M$ is transversal to the vertical Killing vector field $\partial_z$ at every interior point. This is equivalent to the fact that the angle function $\nu=\langle\eta,\partial_z\rangle$ has no zeros at any interior point of the graph. As a matter of fact, each vertical $H$-graph in an $\Ekt$ space is a stable surface, since either the function $\nu$ or $-\nu$ is positive.

\section{Height estimates}

In this section, $H$ will denote a positive constant and $\big(\E,\langle\cdot,\cdot\rangle\big)$ will be either the space $\nil$ or $\psl$ with the corresponding metric. In particular, as in both spaces we have $\kappa\leq 0$, $\E$ is given by the global model defined in Section \ref{prel}. The theorem that we prove is the following:

\begin{teo}\label{teoprin}
Let $H$ be a positive constant and suppose that
$$
4H^2+\kappa>0.
$$
Then, there exists a constant $C=C(H,\kappa,\tau)>0$, such that for every vertical $H$-graph $M$ in $\E$ whose positive height is realized and with boundary contained in a section, the height that $M$ reaches over that section is at most $C$.
\end{teo}

In the space $\psl$, the hypothesis $4H^2+\kappa>0$ has a relevant geometric sense, since this condition for $H$ and $\kappa$ ensures the existence of a rotationally symmetric $H$-sphere. In general, in an $\Ekt$ space the quantity $\sqrt{-\kappa}/2$ is known as the \emph{critical mean curvature}. There exists an $H$-sphere in an $\Ekt$ space if and only if $H>\sqrt{-\kappa}/2$.

Before proving Theorem \ref{teoprin} we recall a technical Lemma that guarantees a uniform bound of the second fundamental form for $H$-graphs in $\E$. See \cite{RST} for a detailed proof.

\begin{lem}\label{estuniforme}
Let $M$ be a vertical $H$-graph in $\E$, with boundary of $M$ in $\{z=0\}$. If $d>0$, there is a constant $K$, depending on $d$ and $\E$, such that $|\sigma(p)|<K$ for all $p$ in $M$ with $d(p,\partial M)>d$.
\end{lem}

Now, we stand in position to prove Theorem \ref{teoprin}.

\begin{proof}
Arguing by contradiction, suppose that the height estimate in the statement of the theorem does not hold. Then, there exists a sequence of compact vertical $H$-graphs $M_n$, whose boundaries are contained in sections of the form $\{z=z_n\}$, and such that if we denote by $h_n$ the height of each $M_n$ from $\{z=z_n\}$, then $\{h_n\}\rightarrow\infty$. After a vertical translation we can suppose that all the boundaries are contained in the section $\Pi=\{z=0\}$. By the mean curvature comparison principle, each graph is contained in one of the half-spaces $\{z\geq 0\}$ or $\{z\leq 0\}$. Passing to a subsequence we can suppose that all the graphs lie above $\Pi$, i.e. they lie in the half-space $\{z\geq 0\}$. Let $\eta_n$ be the unit normal to each $M_n$ such that the mean curvature with respect to $\eta_n$ is $H$. In particular, each $M_n$ is downwards oriented as a consequence again of the mean curvature comparison principle, and thus every angle function $\nu_n=\langle\eta_n,\partial_z\rangle$ is a negative function on $M_n$. Fix some positive number $d$ and let us now denote $M_n^*:=\{p\in M_n;\ d(p,\partial M_n)>2d\}$. As the heights of $M_n$ from $\Pi$ tend to infinity, it is clear that $M_n^*$ is a non-empty, possibly non-connected, graph over $\Pi$ for $n$ large enough. In this situation, Lemma \ref{estuniforme} ensures us that there exists a positive constant $\Lambda$ in such a way that the second fundamental form $\sigma_{M_n^*}$ of each surface $M_n^*$ satisfies $|\sigma_{M_n^*}|<\Lambda$.

Consider for each $n$ the connected component $M_n^0$ of $M_n^*$ of maximum height from $\Pi$. Let $x_n\in M_n^0$ be the point where this maximum height is attained, and consider the isometry $\Phi_n$ that sends $x_n$ to the origin. Now, define $M_n^1=\Phi_n(M_n^0)$. The length of the second fundamental form of each graph $M_n^1$ is uniformly bounded by $\Lambda>0$, as all the $M_n^1$ are obtained by translations of subsets of $M_n^*$. Moreover, the distances in $M_n^1$ of the origin to $\parc M_n^1$ diverge to $\8$. Now, by a standard compactness argument for a sequence of surfaces with bounded curvature, we deduce that, up to a subsequence, there are subsets $K_n\subset M_n^1$ that converge uniformly on compact sets in the $C^2$ topology to a complete, possibly non-connected, $H$-surface $M_{\8}$ that passes through the origin. From now on, we will consider the connected component of $M_{\8}$ that passes through the origin, and we will still denote this component by $M_{\8}$ . Let $\nu_{\8}:=\esiz \eta_{\8},\partial_z\esde$ denote the angle function of $M_{\8}$, where here $\eta_{\8}$ is the unit normal of $M_{\8}$. Since $M_{\8}$ is a limit of the downwards-oriented graphs $M_n^1$, we see that $\nu_{\8}$ is non-positive. We claim that $\nu_{\8}$ cannot be bounded away from zero; indeed, assume that $\nu_{\8}^2\geq c>0$ for some $c>0$. Consider the projection $\mathfrak{p}: M_{\8}\rightarrow\M^2(\kappa)$, let $\langle , \rangle_{\mathrm{proj}}$ be the induced metric on $M_{\8}$ via $\mathfrak{p}$, and let $\langle , \rangle$ be the induced ambient metric on $M_{\8}$. 

As $\esiz,\esde$ is complete and it is well-known that $\nu_{\8}^2\langle , \rangle\leq\langle , \rangle_{\mathrm{proj}}$, we conclude by $\nu_{\8}^2\geq c>0$ that  $\langle , \rangle_{\mathrm{proj}}$ is also complete. In particular, $\mathfrak{p}$ is a local isometry from $(M_{\8},\esiz,\esde_{\mathrm{proj}})$ onto $\M^2(\kappa)$. In these conditions, $\mathfrak{p}$ is necessarily a (surjective) covering map over the simply connected surface $\M^2(\kappa)$, and thus $M_\infty$ is an entire vertical graph. Let $\cS$ be the sphere with constant mean curvature $H$; the condition $4H^2+\kappa>0$ ensures us the existence of such a sphere for the case $\kappa<0$. Let $\cS(0)$ be such a sphere centered at the origin. Translate $\cS(0)$ down until it is below the graph of $M_\infty$. Then translate the sphere back up until it touches $M_\infty$ for the first time. By the maximum principle the sphere equals $M_\infty$, which contradicts that $M_\infty$ is not compact. Therefore, there must exist a sequence of $p_n\in M_{\8}$ with $\nu_{\8}(p_n)\to 0$. 

Let $\Theta_n$ be an isometry in $\E$ that takes each point $p_n$ to the origin $\mathfrak{o}\in\E$, and define $M_\8^n=\Theta_n(M_\8)$, which is a sequence of complete, stable surfaces with constant mean curvature $H$ passing through $\mathfrak{o}$ and whose angle functions satisfy $\nu_\8^n\leq 0$. Again, standard elliptic theory ensures that, up to a subsequence, the surfaces $M_\8^n$ converge to a stable $H$-surface $M_\8^*$, passing through $\mathfrak{o}$. As this convergence is $C^2$, the angle function $\nu_\8^*$ of $M_\8^*$ satisfies $\nu_\8^*\leq 0$ and $\nu_\8^*(\mathfrak{o})=0$. Also, the stability operators $\mathcal{L}_n$ converge to the stability operator $\mathcal{L}_\infty$ of the limit surface $M_\infty^*$.

The maximum principle for elliptic operators applied to $\mathcal{L}_\infty$ yields that any non-zero solution to \eqref{pdeangle} changes sign around any of its zeros. As $\mathcal{L}_\infty$ also admits the zero function as a solution and $\nu^*_\infty$ vanishes at $\mathfrak{o}$, the condition $\nu_\8^*\leq 0$ implies that $\nu_\8^*$ is identically zero. Therefore the limit surface $M_{\8}^*$ is contained in a flat cylinder $\gamma \times \R$, for a planar curve $\gamma$ in $\R^2$ or $\H^2$ (depending on whether $\kappa=0$ or $\kappa<0$, respectively). An analytic prolongation argument yields that the maximal surface containing $M_\8^*$ has to be the complete flat cylinder $\gamma \times \R$. This cylinder is an $H$-cylinder as well, and thus the geodesic curvature of $\gamma$ satisfies $\kappa_\gamma=2H$. This implies that $\gamma$ is a closed curve in $\R^2$ or $\H^2$ (depending if $\kappa=0$ or $\kappa<0$, respectively). In the cylinder $\gamma\times\R$, the operator $\mathcal{L}_\infty$ has the expression
$$
\mathcal{L}_\8=\Delta_M+\kappa_\gamma^2+\kappa.
$$

As all the surfaces $M_\8^n$ are stable, the limit cylinder $M_\8^*$ is also a stable surface. But a complete, vertical $H$-cylinder in an $\Ekt$ is stable if and only if \cite{MaPR}
$$
\kappa_\gamma^2+\kappa\leq 0.
$$
Thus, the limit cylinder is stable if and only if $4H^2+\kappa\leq 0$, which is a contradiction with the hypothesis $4H^2+\kappa>0$. This contradiction completes the proof of Theorem \ref{teoprin}.
\end{proof}

\begin{cor}	
If $H$ is a positive constant such that
$$
4H^2+\kappa>0,
$$
then there do not exist complete vertical $H$-graphs defined over relatively compact domains $\Omega\subset\{z=z_0\}$ in the spaces $\nil$ and $\psl$.
\end{cor}

\begin{proof}
Let $M$ be a complete vertical $H$-graph over a relatively compact domain $\Omega\subset\{z=z_0\}$. Without losing generality we can suppose that $M$ lies in the half-space $\{z\leq 0\}$ and intersects tangentially the section $\{z=0\}$. Let $C$ be the constant appearing in Theorem \ref{teoprin}. Then, as the height of $M$ with respect to the section $\{z=0\}$ is unbounded, there exists some $d_0>0$ such that if we intersect $M$ with the half-space $\{z\geq -d_0\}$, we obtain a compact $H$-graph with boundary lying in the section $\{z=-d_0\}$ and with height over $\{z=-d_0\}$ greater than $C$, contradicting Theorem \ref{teoprin}.
\end{proof}

We finish this paper with two observations concerning further discussions of height estimates of $H$-graphs in $\nil$ and $\psl$.

First, although the constant $C$ in Theorem \ref{teoprin} is not explicit, for some values of $H$ we can derive an estimate for it. Let $S:\E\rightarrow\R$ denote the \emph{scalar curvature} of $\mathbb{E}$, and suppose that there exists some constant $c>0$ such that the inequality 
\begin{equation}\label{ineqescalar}
3H^2+S(x)\geq c
\end{equation}
holds for every $x\in\E$. It was proved by Rosenberg in \cite{Ro2} that if $\Sigma$ is a stable $H$-surface immersed in $\E$, for every $p\in\Sigma$ one has 
$$
d_\Sigma(p,\partial\Sigma)\leq \frac{2\pi}{\sqrt{3c}}.
$$
Recall also that if $\Sigma$ is an immersed surface in $\E$, the intrinsic distance $d_\Sigma$ is always less or equal than the ambient distance. Thus, whenever inequality \eqref{ineqescalar} holds for some $c>0$ and every $x\in\E$, the height of an $H$-graph is less or equal than $2\pi/\sqrt{3c}$. In particular, the constant $C$ in Theorem \ref{teoprin} can be bounded from above by $2\pi/\sqrt{3c}$.

Second we point out that $2\pi/\sqrt{3c}$ is not optimal. Indeed, denote by $\mathcal{S}^{H,\kappa,\tau}$ to the rotationally symmetric $H$-sphere in $\nil$ or $\psl$, and by $\overline{\mathcal{S}^{H,\kappa,\tau}_+}$ to the upper, closed $H$-hemisphere. For $H$ big enough, the height of $\overline{\mathcal{S}^{H,\kappa,\tau}_+}$ tends to zero: see \cite{Tom,Tor} for explicit expressions of the height of $\overline{\mathcal{S}^{H,\kappa,\tau}_+}$ in the spaces $\nil$ and $\psl$. But for $H$ large enough inequality \eqref{ineqescalar} holds, proving that the estimate $2\pi/\sqrt{3c}$ is not sharp.

Motivated by the discussions made in the Introduction about the height estimates for $H$-graphs in the space forms $\R^3$ and $\mathbb{H}^3$, and in the product spaces $\M^2(\kappa)\times\R$, we suggest that the maximum height that an $H$-graph $M$ should attain in both $\nil$ and $\psl$ is the height of the upper, closed $H$-hemisphere $\overline{\mathcal{S}^{H,\kappa,\tau}_+}$, with equality at some point if and only if $M$ agrees with $\overline{\mathcal{S}^{H,\kappa,\tau}_+}$.

\vspace{1cm}

{\bf Acknowledgements} The author wants to express his gratitude to the referee for helpful comments and observations.

\def\refname{References}

\noindent The author was partially supported by MICINN-FEDER Grant No. MTM2016-80313-P, Junta de Andalucía Grant No. FQM325 and FPI-MINECO Grant No. BES-2014-067663.


\begin{thebibliography}{9}

\bibitem[AbRo1]{AbRo1} U. Abresch, H. Rosenberg, A Hopf differential for constant mean curvature surfaces in $\S^2\times\R$ and $\H^2\times\R$, {\it  Acta Math.} {\bf 193} (2004), 141--174.

\bibitem[AbRo2]{AbRo2} U. Abresch, H. Rosenberg, Generalized Hopf differentials, \emph{Mat. Contemp.} \textbf{28} (2005), 1--28.

\bibitem[AEG]{AEG} J. A. Aledo, J. M. Espinar, J. A. Gálvez, Height estimates for surfaces with positive constant mean curvature in $\M^2\times\R$, {\it Illinois J. Math.} {\bf 52} (2008), 203--211.

\bibitem[BCE]{BCE} J. L. Barbosa, M. do Carmo, J. Eschenburg, Stability of hypersurfaces with constant mean curvature in Riemannian manifolds, {\it Math Z.} {\bf 197} (1988), 123--138.

\bibitem[Da]{Da} B. Daniel, Isometric immersions into 3-dimensional homogeneous manifolds, {\it Comment. Math. Helv.} \textbf{82} (2007), 87--131.

\bibitem[DHM]{DHM} B. Daniel, L. Hauswirth, P. Mira, Constant mean curvature surfaces in homogeneous manifolds, Korea Institute for Advanced Study, Seoul, Korea, 2009.

\bibitem[FeMi]{FeMi} I. Fern\'andez, P. Mira, Constant mean curvature surfaces in 3-dimensional Thurston geometries. In \emph{Proceedings of the International Congress of Mathematicians}, Volume II (Invited Conferences), pages 830--861. Hindustan Book Agency, New Delhi, 2010.

\bibitem[Fi]{Fi} D. Fischer-Colbrie, On complete minimal surfaces with finite Morse index in three-manifolds, {\it Invent. Math.} {\bf 82} (1985), 121--132.

\bibitem[He]{He} E. Heinz, On the nonexistence of a surface of constant mean curvature with finite area and prescribed rectifiable boundary, {\it Arch. Ration. Mech. Anal.} {\bf 35} (1969), 249--252. 

\bibitem[HLR]{HLR} D. Hoffman, J. De Lira, H. Rosenberg, Constant mean curvature surfaces in $\M^2\times\R$, {\it  Trans. Amer. Math. Soc.} {\bf 358} (2006), 491--507. 

\bibitem[KKMS]{KKMS} N. Korevaar, R. Kusner, W. H. Meeks III, B. Solomon, Constant mean curvature surfaces in hyperbolic space, {\it Amer. J. Math.} {\bf 114} (1992), 1--43. 

\bibitem[KKS]{KKS} N. Korevaar, R. Kusner, B. Solomon, The structure of complete embedded surfaces with constant mean curvature, {\it J. Differential Geom.} {\bf 30} (1989), 465--503. 

\bibitem[Me]{Me} W. H. Meeks III, The topology and geometry of embedded surfaces of constant mean curvature, {\it J. Differential Geom.} {\bf 27} (1988), 539--552.

\bibitem[MaPR]{MaPR} J.M. Manzano, J. Pérez, M. Rodríguez, Parabolic stable surfaces with constant mean curvature, \emph{Calc. Var. Partial Differential Equations} \textbf{42} (2011), 137--152.

\bibitem[Ro1]{Ro1} H. Rosenberg, Hypersurfaces of constant curvature in space forms, \emph{Bull. Sci. Math.} {\bf 117} (1993), 211--239.

\bibitem[Ro2]{Ro2} H. Rosenberg, Constant mean curvature surfaces in homogeneously regular 3-manifolds, \emph{Bull. Austral. Math. Soc.}, \textbf{74} (2006), 227--238.

\bibitem[RoSa]{RoSa} H. Rosenberg, R. Sa Earp, The geometry of properly embedded special surfaces in $\R^3$, e.g., surfaces satisfying $aH+bK=1$, where $a$ and $b$ are positive, \emph{Duke Math. J.} {\bf 73} (1994), 291--306.

\bibitem[RST]{RST} H. Rosenberg, R. Souam, E. Toubiana, General curvature estimates for stable $H$-surfaces in 3-manifolds and applications, {\it J. Differential Geom.} {\bf 84} (2010), 623--648. 

\bibitem[Se]{Se} J. Serrin, On surfaces with constant mean curvature which span a given space curve, \emph{Math. Z.} \textbf{112} (1969), 77--88.

\bibitem[Tom]{Tom} P. Tomter,  Constant mean curvature surfaces in the Heisenberg group. \emph{Proc. Sympos. Pure Math.} \textbf{54} (1993), 485--495.

\bibitem[Tor]{Tor} F. Torralbo, Rotationally invariant constant mean curvature surfaces in homogeneous 3-manifolds. \emph{Diff. Geo. Appl.} \textbf{28} (2010), 593--607. 
\end{thebibliography}
\end{document}